\documentclass[12pt,reqno]{amsart}

\usepackage[T1]{fontenc}
\pdfoutput=1
\usepackage{mathptmx,microtype,hyperref,graphicx,dsfont}
\usepackage{amsthm,amsmath,amssymb}
\usepackage[foot]{amsaddr}
\usepackage{enumitem}
\usepackage[nameinlink,capitalise]{cleveref}
\usepackage[font=footnotesize,margin=3.2cm]{caption}
\usepackage{cite}
\usepackage{todonotes}

\crefname{theorem}{Theorem}{Theorems}
\crefname{thm}{Theorem}{Theorems}
\crefname{mainthm}{Theorem}{Theorems}
\crefname{lemma}{Lemma}{Lemmas}
\crefname{lem}{Lemma}{Lemmas}
\crefname{remark}{Remark}{Remarks}
\crefname{claim}{Claim}{Claims}
\crefname{subclaim}{Sub-claim}{Sub-claims}
\crefname{prop}{Proposition}{Propositions}
\crefname{proposition}{Proposition}{Propositions}
\crefname{defn}{Definition}{Definitions}
\crefname{corollary}{Corollary}{Corollaries}
\crefname{conjecture}{Conjecture}{Conjectures}
\crefname{question}{Question}{Questions}
\crefname{chapter}{Chapter}{Chapters}
\crefname{section}{Section}{Sections}
\crefname{figure}{Figure}{Figures}
\newcommand{\crefrangeconjunction}{--}

\theoremstyle{plain}

\newtheorem{thm}{Theorem}
\newtheorem*{thm*}{Theorem}

\newtheorem{prop}[thm]{Proposition}

\theoremstyle{definition}

\theoremstyle{remark}

\newcommand{\eps}{\varepsilon}
\renewcommand{\P}{{\bf P}}
\newcommand{\E}{{\bf E}}

\newcommand{\G}{{\mathcal G}}
\newcommand{\Gnp}{\G_{n,p}}

\renewcommand{\le}{\leqslant}
\renewcommand{\ge}{\geqslant}

\author[Z. Bartha]{Zsolt Bartha}
\address{Alfr\'{e}d R\'{e}nyi Institute of Mathematics}
\email{bartha@renyi.hu}

\author[B. Kolesnik]{Brett Kolesnik}
\address{Department of Statistics, 
University of Oxford}
\email{brett.kolesnik@stats.ox.ac.uk}

\author[G. Kronenberg]{Gal Kronenberg}
\address{Department of Mathematics, 
University of Oxford}
\email{gal.kronenberg@maths.ox.ac.uk}


\begin{document}

\title
[$H$-percolation with a random $H$]
{$H$-percolation with a random $H$}

\begin{abstract}
In $H$-percolation, we start with an 
Erd\H os--R\'enyi graph 
${\mathcal G}_{n,p}$
and then iteratively add edges that complete  
copies of $H$. 
The process percolates if
all edges missing from ${\mathcal G}_{n,p}$
are eventually added. 
We find the critical threshold
$p_c$ when 
$H={\mathcal G}_{k,1/2}$ is 
uniformly random,  
solving a problem of
Balogh, Bollob\'as and Morris. 
\end{abstract}

\maketitle

\section{Introduction}\label{S_intro}

Following 
Balogh, Bollob\'as and Morris \cite{BBM12}, 
we fix a graph $H$, and 
begin with an  
Erd\H os--R\'enyi graph
$\G_0=\Gnp$. 
Then, for $t\ge1$, 
we obtain $\G_t$ 
from $\G_{t-1}$ 
by adding every edge that creates a new copy of 
$H$. 
We let  $\langle\Gnp\rangle_H=\bigcup_{t\ge0}  \G_t$
denote the graph containing all eventually added edges. 
When $\langle\Gnp\rangle_H = K_n$, we say that 
the process {\it $H$-percolates,}
or equivalently, in the terminology of  
Bollob\'as \cite{B67}, that 
$\Gnp$ is {\it weakly $H$-saturated}. 

$H$-percolation 
generalizes
bootstrap percolation 
\cite{PRK75,CLR79}, which is one of the most well-studied of all  
cellular automata \cite{U50,vN66}. 
Such processes model 
evolving  networks, in which sites
update their status according to the behavior of their neighbors. 
Although the dynamics are local, they can lead to 
global behavior, 
emulating various real-world phenomena
of interest,  
such as  
tipping points,
super-spreading, self-organization
and collective decision-making. 

The {\it critical $H$-percolation threshold} 
is defined to be the point 
\[
p_c(n,H)=\inf\{p>0:\P(\langle\Gnp\rangle_H=K_n)\ge 1/2\}, 
\]
at which $\Gnp$ 
becomes likely 
to $H$-percolate.
Problem 1 in \cite{BBM12} asks for 
\[
\ell(H)=-\lim_{n\to\infty} \frac{\log p_c(n,H)}{\log n}
\]
for {\it every} graph $H$. 
Finding $\ell$ corresponds to locating $p_c=n^{-\ell+o(1)}$ 
up to poly-logarithmic factors.  
The authors state that ``Problem 1 is likely to be hard.'' Indeed, 
e.g., even the cases $\ell(K_{2,t})$ remain open, for all
$t\ge5$, see Bidgoli et al.\ \cite{BMTR18}.

In this note, we answer Problem 1 for {\it random} graphs $H$. 
We find that typically $\ell(H)=1/\lambda(H)$, where 
\[
\lambda(H)=\frac{e_H-2}{v_H-2}
\]
is the adjusted edge per vertex ratio in \cite{BBM12}, 
for graphs $H$ with $v_H$  vertices and $e_H$ 
edges. 
We recall that Theorem 1 in \cite{BBM12} shows that
$\ell(K_k)=1/\lambda(K_k)$ 
for cliques. 
This corresponds to the case $\alpha=1$
in the following result, 
which, in some sense, shows that
$\ell=1/\lambda$
is ``stable.'' 

\begin{thm}
\label{T_main}
Fix  
$0< \alpha \le 1$.
Then, with high probability,  as $k\to\infty$, we have that 
$\ell({\mathcal G}_{k,\alpha})=1/\lambda({\mathcal G}_{k,\alpha})$. 
\end{thm}

The case $\alpha=1/2$ answers 
Problem 6 in 
\cite{BBM12}, that asks for ``bounds on $p_c(n,\G_{k,1/2})$ 
which hold with high probability as $k\to\infty$.'' 
We find that  
$p_c=n^{-1/\lambda+o(1)}$,  
except with exponentially
small probability.
In particular, this holds, almost surely, 
for all large $k$, 
by the Borel--Cantelli lemma.
Note that  
$\G_{k,1/2}$ is uniformly random 
(from the set of simple graphs on $k$ vertices).
In this sense, we find
$\ell(H)$
for ``most'' graphs $H$. 

In fact, our proof works for  
$\alpha\ge A(\log k)/k$, for some 
sufficiently
large constant $A$. 
See \cref{S_smallA} below 
for more on this. 

In closing, we note that
it is natural to ask about 
$k$ growing with $n$. 
For instance, $\lambda\approx k/4$ when $\alpha=1/2$, 
in which case     
$k\ll \log n$ appears to be 
the region of interest. 

\subsection{Acknowledgments} 
ZB is supported by the ERC Consolidator Grant 772466 “NOISE”.
GK is supported by the Royal Commission for the Exhibition of 1851.

\section{Background}
\label{S_back}

As defined in \cite{BBM12}, 
a graph $H$ is 
{\it balanced} if 
\begin{equation}\label{E_bal}
\frac{e_F-1}{v_F-2}\le \lambda(H), 
\end{equation}
for all subgraphs $F\subset H$ with $3\le v_F<v_H$. 
Otherwise, we call $H$ {\it unbalanced}. 
In \cite{BBM12}, it is shown that $\ell \ge 1/\lambda $ for all balanced $H$.
A sharper upper bound (on $p_c$) is proved in \cite{BK23} for 
{\it strictly balanced} graphs $H$, which  satisfy 
the above condition, with $\le$ replaced by $<$. 
Specifically, it is shown that $p_c\le O(n^{-1/\lambda})$
for all such $H$, replacing $n^{o(1)}$ with a constant.

On the other hand, 
in \cite{BK23} it is shown that 
$\ell\le 1/\lambda_*$ for all 
$H$, with $v_H\ge4$ and minimum degree $\delta_H\ge2$, where
\[
\lambda_*(H)=\min \frac{e_H-e_F-1}{v_H-v_F},
\]
minimizing over all subgraphs $F\subset H$ with $2\le v_F<v_H$. 
The quantity $\lambda_*$ is 
related to the ``cost'' of adding an edge, 
via the $H$-percolation dynamics, as depicted in \cref{F_lamstar}.

As observed in \cite{BK23}, we have that $\lambda_*\le \lambda$ in general, and 
$\lambda_*= \lambda$ if and only if $H$ is balanced. 
Moreover, when $H$ is balanced, single edges $F$ (when $v_F=2$) 
attain the minimum $\lambda_*$. 
When $H$ is strictly balanced, these are the only minimizers.

\begin{figure}[h!]\includegraphics[scale=1]{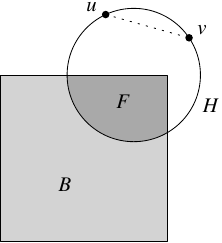}
\caption{
Suppose that edges in a ``base'' graph $B$ 
have already been added. 
Next, we
add an edge $\{u,v\}$
using a copy of $H$. 
The ``price''  
is $e_H-e_F-1$ edges 
and $v_H-v_F$  vertices, where $F=H\cap B$. 
Hence the edge per vertex ``cost'' 
is at least $\lambda_*$. 
}
\label{F_lamstar}
\end{figure}

\section{Proof}

To prove \cref{T_main}, we show that
 $\G_{k,\alpha}$ is 
 (strictly) balanced, 
 with high probability. 
 By the results from \cite{BBM12,BK23} discussed above, 
 the result follows. 
 
Let us give some intuition for why random graphs $\G_{k,\alpha}$, 
with sufficiently large 
edge probability 
$\alpha$, are likely to be balanced. 
If \eqref{E_bal} fails 
for some $F\subset \G_{k,\alpha}$, then 
the edge density of $F$ is larger than that of $\G_{k,\alpha}$. 
For instance, if $F$ includes all but $j$ vertices, then 
the number of edges between these $j$ vertices and $F$
is less than (roughly) $j$ times half the average degree in $F$. 
By the law of large numbers (when $k$ is large) 
this large deviation event is most likely to occur when $j=1$, 
and even then, it is quite unlikely.

We will use the following, standard Chernoff tail estimates. 
Recall that if $X$ is binomial with 
mean $\mu$ then, for $0<\delta<1$, 
\begin{equation}\label{E_Cher}
\P(|X/\mu-1|\ge \delta)
\le 2\exp[-\delta^2\mu/3]. 
\end{equation}
On the other hand, for $\delta\ge1$, 
\begin{equation}\label{E_CherBig}
\P(X\ge(1+\delta)\mu)
\le \exp[-\delta\mu/3]. 
\end{equation}

\begin{prop}
\label{L_ranH_SB}
Suppose that $A(\log k)/k\le \alpha\le 1$, 
for some sufficiently large
constant $A$. 
Then, almost surely, $\G_{k,\alpha}$ is strictly balanced, 
for all large $k$. In particular, 
$\ell({\mathcal G}_{k,\alpha})=1/\lambda({\mathcal G}_{k,\alpha})$, 
for all large $k$.
\end{prop}

\begin{proof}
If $\G_{k,\alpha}$ is not strictly balanced,   
then there is a subset 
$S\subset [k]$ of size $3\le s< k$, 
such that
\begin{equation}\label{E_LD0}
\frac{e_S-1}{s-2}
\ge
\frac{e(\G_{k,\alpha})-2}{k-2},
\end{equation}
where $e_S$ is the number of edges in $\G_{k,\alpha}$ 
with both endpoints in $S$. 
Let $e_S'
=e(\G_{k,\alpha})-e_S$ be the number of all other edges in $\G_{k,\alpha}$, 
with at most one endpoint
in $S$. For any given $S$, the 
random variables
$e_S$ and $e_S'$ are independent. 
Since 
\[
\frac{1}{s-2}-\frac{1}{k-2}
=\frac{k-s}{(s-2)(k-2)},
\]
\eqref{E_LD0} implies that 
\begin{equation}\label{E_LD}
\frac{e_S-1}{s-2}
\ge\frac{e_S'-1}{k-s}.
\end{equation}

Using the Chernoff bounds 
\eqref{E_Cher} and \eqref{E_CherBig}, we will estimate the probability
that \eqref{E_LD} occurs 
for some $S$ of size $3\le s<k$.
Note 
that 
\begin{equation}\label{E_Es}
\E(e_S)=\frac{s(s-1)}{2}\alpha,
\quad 
\E(e_S')=\frac{(k-s)(k+s-1)}{2}\alpha.
\end{equation}
Therefore, if 
$e_S=(1+\delta)\E(e_S)$ then, in order for 
\eqref{E_LD} to hold, we would require  
$e_S'\le (1-\delta'+o(1))\E(e_S')$, where
\begin{equation}\label{E_deltaprime}
\delta'=\frac{k-\delta s}{k+s}.
\end{equation}
Let $P_s(\delta)$ denote the probability that $e_S$ and $e_S'$
take such values. 
(We assume $\delta>-1$ throughout,  since clearly 
the right hand side of 
\eqref{E_LD0} will be positive, with high probability.)

{\it Case 1.} First, we consider the case that $\delta<1$.  
By \eqref{E_Cher} and \eqref{E_Es}, 
for any given $S$, 
it follows that 
$e_S$ and $e_S'$ take such values 
with 
probability at most 
$\exp[-\eps\alpha\vartheta(\delta)]$, 
for some $\eps>0$, 
where
\[
\vartheta(\delta)=\delta^2 s^2+(\delta')^2(k^2-s^2). 
\]

Note that $\vartheta(\delta)$ is convex and minimized at 
$\delta_*=(k-s)/2s$, at which point  
$\vartheta(\delta_*)=(k-s)k/2$. 
Therefore, when $\delta<1$, 
\begin{equation}\label{E_case1}
{k\choose k-s}P_s(\delta)\le  (ke^{-\eps\alpha k})^{k-s},
\end{equation}
for some $\eps>0$. 

{\it Case 2.} Next, we consider $\delta\ge1$. 
We claim that 
\begin{equation}\label{E_case2}
{k\choose s}P_s(\delta)
\le (k e^{-\eps\alpha k})^s,
\end{equation}
for some $\eps>0$. 
We will consider the cases that
$k/2s$ is smaller/larger than $\delta$
separately. In the former case, 
we consider only the large deviation by $e_S$,
and in the latter case, 
only the deviation by 
$e_S'$.

{\it Case 2a.} If $k/2s\le \delta$, 
then consider the event  
$e_S=(1+\delta)\E(e_S)$. By \eqref{E_CherBig} and \eqref{E_Es}, this occurs
with probability at most $\exp[-\eps\alpha s^2(k/s)]$, 
for some $\eps>0$. Therefore \eqref{E_case2} holds in this case. 

{\it Case 2b.} If $k/2s > \delta$,  
then consider the event  
$e_S'\le (1-\delta'+o(1))\E(e_S')$, with $\delta'$ as in \eqref{E_deltaprime}.
Since $1\le \delta< k/2s$, we have $\delta'>1/3$. 
By \eqref{E_Cher} and \eqref{E_Es},
this occurs with probability 
at most $\exp[-\eps\alpha(k^2-s^2)]$, 
for some $\eps>0$. Then, since $s< k/2$, we find, once
again, that  \eqref{E_case2} holds. 

Finally, we take a union bound, summing over all relevant 
values $3\le s<k$ for the size of $S$. 
For each such $S$, there are $O(s^2)$ possible values of $e_S$. 
Therefore, combining \eqref{E_case1} and \eqref{E_case2} above, 
we find that 
\eqref{E_LD} holds for some such $S$ with probability at most 
$O(k^3e^{-\eps\alpha k})$, 
for some $\eps>0$. 
This is $\ll1$ for 
$\alpha\ge A(\log k)/k$, 
for large $A$, 
in which case 
$\G_{k,\alpha}$ is strictly balanced with high probability. 
Furthermore, for large $A$, this probability is summable. 
Therefore,  by the Borel--Cantelli lemma, 
almost surely, we have that $\ell=1/\lambda$, for all large $k$.
\end{proof}

\subsection{Smaller $A$}
\label{S_smallA}
For ease of exposition
(and since, 
in our view, $\alpha=1/2$
is the most interesting case)  
we have not pursued 
the smallest possible 
constant $A$ in the proof above.  
However, we expect that more technical arguments 
can show that, with high probability, 
${\mathcal G}_{k,\alpha}$ is balanced, 
and so $\ell=1/\lambda$, provided that $\alpha=A(\log k)/k$
with $A> 2/\log(e/2)$. 

Let us give  
some brief intuition in this direction. 
Recall that, in the proof above,  
the case $s=k-1$ is critical. This corresponds (roughly speaking)
to the existence
of a vertex $v$ with degree less than half the average degree
in the subgraph $F\subset {\mathcal G}_{k,\alpha}$ induced
by the other $k-1$ vertices. 
Applying the sharper,  relative entropy 
tail bound for the binomial (see, e.g., 
Diaconis and Zabell
\cite[Theorem 1]{DZ91}) when   
$\alpha = A (\log k)/k$, we have that   
\[
\P({\rm Bin}(k,\alpha)\le \alpha k/2)
\le O\left(\frac{k^{-(A/2)\log (e/2)}}{\sqrt{\log k}}\right).
\]
Large deviations of $e_F$ 
are significantly less likely, as the 
expected number of edges in $F$ is of a larger order $O(\alpha k^2)$
than that $O(\alpha k)$ of the expected degree of $v$. 
Therefore, when $A= 2/\log(e/2)$, 
it can be seen that 
such a vertex $v$ exists with probability at most 
$O(1/\sqrt{\log k})\ll1$. 
Therefore, if \eqref{E_bal} fails, it is due to some smaller $F$, 
with $s\le k-2$ vertices, 
however,  
the existence of such an $F$ is 
increasingly less likely as $s$ decreases. 
On the other hand, when $A< 2/\log(e/2)$
it can be shown, by the second moment method, 
that with high probability such a vertex $v$ exists, 
in which case 
${\mathcal G}_{k,\alpha}$ is unbalanced. 

In closing, we remark that 
it might be of interest to  study the behavior of 
$\ell$ as $\alpha$ 
decreases to 
the point $\alpha\sim(\log k)/k$ of connectivity. 
Note that $2/\log(e/2)\approx 6.518$. 
When 
the minimum degree of 
${\mathcal G}_{k,\alpha}$
is at least $2$, the general
upper bound 
$\ell\le 1/\lambda_*$
from \cite{BK23} holds.  
In the extreme case that  
${\mathcal G}_{k,\alpha}$ has a leaf, 
the value
of $\ell$ is given by Proposition 26
in \cite{BBM12}.

\makeatletter
\renewcommand\@biblabel[1]{#1.}
\makeatother
\providecommand{\bysame}{\leavevmode\hbox to3em{\hrulefill}\thinspace}
\providecommand{\MR}{\relax\ifhmode\unskip\space\fi MR }
\providecommand{\MRhref}[2]{%
  \href{http://www.ams.org/mathscinet-getitem?mr=#1}{#2}
}
\providecommand{\href}[2]{#2}


\begin{thebibliography}{1}

\bibitem{BBM12}
J.~Balogh, B.~Bollob{\'a}s, and R.~Morris, \emph{Graph bootstrap percolation},
  Random Structures Algorithms \textbf{41} (2012), no.~4, 413--440.

\bibitem{BK23}
Z.~Bartha and B.~Kolesnik, \emph{Weakly saturated random graphs}, preprint
  available at \href{https://arxiv.org/abs/2007.14716}{arXiv:2007.14716}.

\bibitem{BMTR18}
M.~R. Bidgoli, A.~Mohammadian, and B.~Tayfeh-Rezaie, \emph{On
  {$K_{2,t}$}-bootstrap percolation}, Graphs Combin. \textbf{37} (2021), no.~3,
  731--741.

\bibitem{B67}
B.~Bollob\'{a}s, \emph{Weakly {$k$}-saturated graphs}, Beitr\"age zur
  {G}raphentheorie ({K}olloquium, {M}anebach, 1967), Teubner, Leipzig, 1968,
  pp.~25--31.

\bibitem{CLR79}
J.~Chalupa, P.~L. Leath, and G.~R. Reich, \emph{Bootstrap percolation on a
  {B}ethe lattice}, J. Phys. C \textbf{21} (1979), L31--L35.

\bibitem{DZ91}
P.~Diaconis and S.~Zabell, \emph{Closed form summation for classical
  distributions: variations on a theme of de {M}oivre}, Statist. Sci.
  \textbf{6} (1991), no.~3, 284--302.

\bibitem{PRK75}
M.~Pollak and I.~Riess, \emph{Application of percolation theory to 2d-3d
  {H}eisenberg ferromagnets}, Physica Status Solidi (b) \textbf{69} (1975),
  no.~1, K15--K18.

\bibitem{U50}
S.~Ulam, \emph{Random processes and transformations}, Proceedings of the
  {I}nternational {C}ongress of {M}athematicians, {V}ol. 2, {C}ambridge,
  {M}ass., 1950, pp.~264--275.

\bibitem{vN66}
J.~von Neumann, \emph{Theory of self-reproducing automata}, University of
  Illinois Press, Champaign, IL, USA, 1966.

\end{thebibliography}
\end{document}